\documentclass[12pt]{article}%
\usepackage{amsmath}
\usepackage{tabularx}
\usepackage{amsthm}
\usepackage{xcolor}
\usepackage{verbatim}
\usepackage{amsfonts}
\usepackage{amssymb}
\usepackage{graphicx}%
\allowdisplaybreaks
\usepackage[colorlinks=false,linkbordercolor=red]{hyperref}
\hypersetup{%
  colorlinks=false,
  linkbordercolor=red,
  pdfborderstyle={/S/U/W 1}
}

\newtheorem{theorem}{Theorem}

\newtheorem{lemma}[theorem]{Lemma}

\newtheorem*{openproblem}{Open Problem}
\theoremstyle{definition} 
\newtheorem{definition}[theorem]{Definition}
\newtheorem{corollary}[theorem]{Corollary}

\newtheorem{remark}{Remark}

\def\An{\{A_n^i\}_{n=0}^{\infty}}

\def\xn{\{x_n\}_{n=-t}^{\infty}}
\def\Zed{\mathbf{Z}}
\def\ZedP{\mathbf{Z}^+}
\def\Pos{=1,2,\ldots}
\def\Nats{\mathbf{N}}
\def\aftermath{\par\vspace{-\belowdisplayskip}\vspace{-\parskip}\vspace{-\baselineskip}}

\begin{document}

\title{On the Boundedness of Positive Solutions of the Reciprocal Max-Type Difference Equation \\
$\displaystyle{x_{n}=\max\left\{\frac{A^{1}_{n-1}}{x_{n-1}}, \frac{A^{2}_{n-1}}{x_{n-2}}, \ldots, \frac{A^{t}_{n-1}}{x_{n-t}}\right\}}$ \\ 
with Periodic Parameters} 

\author{
Daniel W. Cranston
\thanks{Department of Mathematics and Applied Mathematics, Virginia Commonwealth University, \texttt{dcranston@vcu.edu}}
\and
Candace M. Kent
\thanks{Department of Mathematics and Applied Mathematics, Virginia Commonwealth University, \texttt{cmkent@vcu.edu}}
}

\date{}

\maketitle


\begin{abstract}
We investigate the boundedness of positive solutions of the reciprocal max-type difference equation

\[ x_{n}=\max\left\{\frac{A_{n-1}^{1}}{x_{n-1}}, \frac{A_{n-1}^{2}}{x_{n-2}},
 \ldots, \frac{A_{n-1}^{t}}{x_{n-t}}\right\}, \ \ n=1, 2, \ldots, \]

\noindent where, for each value of $i$, the sequence
$\{A_{n}^{i}\}_{n=0}^{\infty}$ of positive numbers is periodic with period
$p_{i}$.  We give both sufficient conditions on the $p_{i}$'s for the
boundedness of all solutions and sufficient conditions for all solutions to be
unbounded.  This work essentially complements the work by Biddell and Franke, 
who showed that as long as every positive solution of our equation is
\emph{bounded}, then every positive solution is eventually periodic, thereby
leaving open the question as to when solutions are bounded. 

\vspace{.1in} 


\end{abstract}



\section{Introduction}
\label{section:Introduction}

\noindent Difference equations with the maximum function and reciprocal arguments, or reciprocal max-type difference equations, were conceived 
circa 1994, when G. Ladas at the University of Rhode Island in a seminar on difference equations and applications took a variant of the Lozi map,

\[ y_{n}=|y_{n-1}|-y_{n-2}, \ \ n\Pos ,\]

\noindent made a change of variables, $y_{n}=\ln x_{n}$, and produced the following equation:

\[ x_{n}=\frac{\max\{x_{n-1}^{2}, 1\}}{x_{n-1}x_{n-2}}, \ \ n\Pos\,. \]

\noindent This equation was extended to the equation

\begin{equation}
x_{n}=\frac{\max\{x^{k}_{n-1}, A\}}{x^{\ell}_{n-1}x_{n-2}}, \ \ n\Pos ,
\label{equation:one}
\end{equation}

\noindent where $k, \ell\in\mathbf{Z}$, $A\in (0, \infty)$, and initial
conditions are positive (cf.~\cite{LADAS}).  Investigation
of (\ref{equation:one}) with $A=1$ led to interesting results 
about the periodicity of solutions, e.g., every solution is periodic with period 9
if $k=2$ and $\ell =1$.

\vspace{.1in}

\noindent From 1994 to the present, difference equations with the maximum (or minimum) function
and reciprocal arguments have rapidly evolved into a diverse family of equations.  
In 1998, Al-Amleh, Hoag, and Ladas~\cite{AHL} investigated one of the earliest autonomous reciprocal max-type equations,

\begin{equation}
x_{n}=\max\left\{\frac{a}{x_{n-1}}, \frac{A}{x_{n-2}}\right\}, \ \ n\Pos,
\label{equation:two}
\end{equation}

\noindent where $a, A\in\mathbf{R}-\{0\}$.  One major result they
obtained was that when $a=1$, $A\in (0, \infty)$, and initial conditions are
positive, every solution is periodic with

\begin{enumerate}

\item period 2 if $A\in (0, 1)$;

\item period 3 if $A=1$;

\item period 4 if $A\in (1, \infty)$.

\end{enumerate}

\noindent However, they also showed that every solution is \emph{unbounded} when $a\neq A$, 
$a, A\in (-\infty, 0)$, and $x_{-1}, x_{0}\in\mathbf{R}-\{0\}$.

\vspace{.1in}

\noindent Replacing the constant coefficient $A$ by the variable coefficient $A_{n}$ (and replacing $a$ by $1$) in 
(\ref{equation:two}), in 1997 Briden et al.~\cite{BGLM} studied the resulting nonautonomous equation

\begin{equation}
x_{n}=\max\left\{\frac{1}{x_{n-1}}, \frac{A_{n-1}}{x_{n-2}}\right\}, \ \ n\Pos, 
\label{equation:three}
\end{equation} 

\noindent where $\{A_{n}\}_{n=0}^{\infty}$ is a periodic sequence of positive
numbers with period 2  
%
%
and initial conditions are positive.  They showed that every positive solution is eventually periodic with

\begin{enumerate}

\item period 2 if $A_{0}A_{1}\in (0, 1)$;

\item period 6 if $A_{0}A_{1}=1$;

\item period 4 if $A_{0}A_{1}\in (1, \infty)$.

\end{enumerate}

\noindent At this point, no unbounded solutions were known.
%
%
A few years later,
Briden et al.~\cite{BGKL} and Grove
et al.~\cite{GKLR} changed the period of 
$\{A_{n}\}_{n=0}^{\infty}$
in (\ref{equation:three}) to period 3, and unbounded solutions made their
first appearance with a positive periodic parameter and positive initial
conditions in a reciprocal max-type equation.  Specifically, letting
$\{A_{n}\}_{n=0}^{\infty}$ be positive and periodic with period 3,
%
%
they 
showed that every solution is

\begin{enumerate}

\item 
eventually periodic with period 2,
if $A_{n}\in (0, 1)$ for all $n\geq 0$; 

\item 
eventually periodic with period 12,
if $A_{n}\in (1, \infty)$ for all $n\geq 0$; 

\item 
\emph{unbounded},
if $A_{i+1}<1<A_{i}$ for some $i\in\{0, 1, 2\}$; 

\item 
eventually periodic with period 3\ in all other cases. 

\end{enumerate}

\vspace{.1in}

\noindent Upon the discovery that unbounded solutions could indeed occur with
the reciprocal max-type equation, Kent and Radin~\cite{KENTR} in 2003 sought
necessary and sufficient conditions for boundedness with the equation 

\begin{equation}
x_{n}=\max\left\{\frac{A_{n-1}}{x_{n-1}}, \frac{B_{n-1}}{x_{n-2}}\right\}, \ \
n\Pos, 
\label{equation:four}
\end{equation}

\noindent where $\{A_{n}\}_{n=0}^{\infty}$ and $\{B_{n}\}_{n=0}^{\infty}$ are
positive periodic sequences with minimal periods $p$ and $q$, respectively,
and initial conditions are positive.  
They showed that every positive solution is:

\begin{enumerate}

\item 
bounded, if neither $p$ nor $q$ is a multiple of 3; 

\item 
unbounded,
if $p=3k$ for some $k\in\ZedP$, such that for some $i\in\{1, 2, 3\}$ and for all $j=1, 2, \ldots, k$,

\[ A_{i+3j}<B_{1}, B_{2}, \ldots, B_{q}<A_{i+3j+1}; \]


\item 
unbounded,
if $q=3k$ for some $k\in\ZedP$, such that for some $i\in\{1, 2, 3\}$ and for all
$j=1, 2, \ldots, k$,

\[ B_{i+3j}<A_{1}, A_{2}, \ldots, A_{q}<B_{i+3j+1}. \]


\end{enumerate}

\vspace{.1in}

\noindent In contrast, Kerbert and Radin~\cite{KERBERTR} showed in 2008 
that every positive solution of the equation

\begin{equation}
x_{n}=\max\left\{\frac{A_{n-1}}{x_{n-1}}, \frac{B_{n-1}}{x_{n-3}}\right\}, \ \
n\Pos, 
\label{equation:five}
\end{equation}

\noindent is unbounded if $\{A_{n}\}_{n=0}^{\infty}$ and
$\{B_{n}\}_{n=0}^{\infty}$ are positive periodic sequences with minimal periods
$p$ and $q$, respectively, and $p$ or $q$ is a multiple of 4 (together with
certain other conditions on $A_{n}$ and $B_{n}$ which we omit for the sake of
brevity).  Furthermore, for the equation

\begin{equation}
x_{n}=\max\left\{\frac{A_{n-1}}{x_{n-2}}, \frac{B_{n-1}}{x_{n-3}}\right\}, \ \ n\Pos, 
\label{equation:six}
\end{equation}

\noindent where again $\{A_{n}\}_{n=0}^{\infty}$ and $\{B_{n}\}_{n=0}^{\infty}$
are positive periodic sequences with minimal periods $p$ and $q$, respectively,
the methods in this paper can be used to show that every positive solution
is unbounded if either $p$ or $q$ is a multiple of 5 (together with certain
other conditions which we again omit for brevity).

\vspace{.1in}

\noindent Around the same time as Kerbert and Radin's investigation, Bidwell and Franke in a landmark paper~\cite{BF} considered the following equation:

\begin{equation}
x_{n}=\max\left\{\frac{A_{n-1}^{1}}{x_{n-1}}, \frac{A_{n-1}^{2}}{x_{n-2}},
 \ldots, \frac{A_{n-1}^{t}}{x_{n-t}}\right\}, \ \ n\Pos,
\label{equation:eight}
\end{equation}

\noindent where $t\in\{1, 2, \ldots\}$, $\{A_{n}^{i}\}_{n=0}^{\infty}$, for all
$i=1, \ldots, t$, is a nonnegative periodic sequence with period
$p_{i}\in\ZedP$, and initial conditions are positive.  They showed that if
every solution of (\ref{equation:eight}) is \emph{bounded}, then every solution
is eventually periodic.  The question then remained:  \emph{Under what
conditions on the nonnegative periodic parameters is every solution bounded?} 

\vspace{.1in}

\noindent In Sections 3 and 4, we largely answer this question when the periodic parameters,
$\{A_{n}^{i}\}_{n=1}^{\infty}$ for all $i=1, \ldots, t$, are positive. 
Specifically, we find sufficient conditions on the parameters' periods 
such that every solution is bounded and also sufficient conditions on the
parameters' periods and on the parameters' values, in comparison with each other,
such that every solution is unbounded.  

\vspace{.1in}

\noindent 
The past decade has seen investigation of
difference equations that are extensions and generalizations of
(\ref{equation:eight}), as well as difference equations with maxima that have
been inspired by differential equations with maxima and automatic control
theory (cf.~\cite{BH}).  All such difference equations have added an order of
great complexity.  For a sampling of this work, see
the papers by \c{C}inar, Stevi\'{c}, and Yal\c{c}inkaya~\cite{CSY}; 
Iri\u{c}an and Elsayad~\cite{IE}; 
Liu, Yang, and Stevi\'{c} ~\cite{LYS}; 
Qin, Sun, and Xi~\cite{QSX}; 
Sauer~\cite{ST1} and~\cite{ST2}; 
Stevi\'{c} ~\cite{SS}; 
Sun~\cite{SF};  
Touafek and Halim~\cite{TH}; and
Yang, Liu, and Lin~\cite{YLL}. 



\section{Preliminaries}
\label{section:Preliminaries}
%
For convenience, we write $[k]$ to denote the set $\{1,\ldots,k\}$.
Let $\An$, for each $i\in[t]$, be a periodic sequence of positive real numbers
with prime period $p_i$. 
The following 
preliminaries will be useful in the sequel.

\begin{definition}[Boundedness and Persistence]
A positive sequence $\xn$ is \emph{bounded} if there exists a positive constant
$M$ with
$$
0 < x_n \le M~~\mbox{for all $n=-t,\ldots,0,1,\ldots$}
$$
and it \emph{persists} (or is \emph{persistent}) if there exists a positive
constant $m$ with

$$
m\le x_n~~\mbox{for all $n=-t,\ldots,0,1,\ldots$.}
$$
\end{definition}

\begin{remark}
To show that a positive solution of a difference equation
$$
x_n=f(x_{n-1},x_{n-2},\ldots,x_{n-t}),~~~\mbox{$n=0,1,\ldots,$}
$$
is unbounded and does not persist,
we exhibit a subsequence of the solution which diverges
to infinity and a subsequence which converges to zero.
\end{remark}

\begin{definition}[Eventual Periodicity]
A positive sequence $\xn$ is \emph{eventually periodic} if there exists $N\ge
-t$ such that $\{x_n\}_{n=N}^{\infty}$ is periodic. 
\end{definition}

\begin{definition}[Extended Periodicity]
A positive sequence $\xn$ is \emph{extended periodic with period $p$} if 
for all $i\in[p]$ either 
${\lim_{n\to \infty}x_{pn+i}=\infty}$
or
${\lim_{n\to \infty} x_{pn+i}=0}$.
%
%
\end{definition}

\begin{lemma}
If $\xn$ is a positive solution of (\ref{equation:eight}), 
%
then $\xn$ is bounded if and only if it persists.
\end{lemma}
\begin{proof}
We show that persistence  implies boundedness.  The proof 
of the converse
is similar, so we omit it.

Let $\xn$ be a persistent positive solution of equation (\ref{equation:eight}); since it is
persistent, there exists $\epsilon > 0$ such that $x_i > \epsilon$ for all $i$.
Let $\alpha=\max_jA^j$.
Now we get $x_n= \max_{1\le j\le t}\left\{\frac{A^j}{x_{n-j}}\right\} \le
\frac{\alpha}{\epsilon}$.  Thus, $\xn$ is bounded.
%
%
\end{proof}

In the following definition, the sequences $\An$ need not be periodic.

\begin{definition}[Hypothesis (H)]
A set of $t$ sequences of positive real numbers $\An$, $i\in[t]$,
satisfies Hypothesis (H) if there exists $j\in[t+1]$ such that
for all $i\in[t]$ we have
\vspace{-.2in}

%
%
%

%
%
%
\begin{align*}
S_{A^i}&=\sup \{A^i_{(t+1)n+(t+1)+j}:~n=0,1,\ldots\}\\
&< I_{A^{t+1-i}} = \inf\{A^{t+1-i}_{(t+1)n+(t+1-i)+j}:~n=0,1,\ldots\}.
\end{align*}
\end{definition}

\begin{remark}
If $\An$, $i\in[t]$, is periodic with period $t+1$, then Hypothesis (H)
becomes the following:
%
$$
A^i_{(t+1)n+(t+1)+j}<
A^{t+1-i}_{(t+1)n+(t+1-i)+j} 
$$
for some $j\in[t+1]$ and for all $i\in[t]$.
%
%
\end{remark}

%

\vspace{.2in}

\noindent The number $t+1$ is ``special" and plays a central role in the results
of the sequel.  This number is based on an \emph{ansatz}, derived from various observations.
%
%
It is easy to show that every positive solution of the difference equation

\[ x_{n}=\frac{1}{x_{n-t}}, \ \ n\Pos \]

\noindent where $t\in\ZedP$ and initial conditions are positive, is
periodic with period $2t$.  Now observe the following:

\begin{enumerate}

\item For (\ref{equation:four}), the period of either $\{A_{n}\}_{n=0}^{\infty}$ or 
$\{B_{n}\}_{n=0}^{\infty}$ must be a multiple of 3\ in order for every
solution to be unbounded.  Now 3 is the average of 2 and
4, the respective periods of every positive solution of the equations

\[ x_{n}=\frac{1}{x_{n-1}} \ \ \ \mbox{and} \ \ \ x_{n}=\frac{1}{x_{n-2}}. \]

\noindent The right sides of these two equations make up the arguments of (\ref{equation:four}).

\item Similarly, for (\ref{equation:five}), the period of either
$\{A_{n}\}_{n=0}^{\infty}$ or $\{B_{n}\}_{n=0}^{\infty}$ must be a multiple of
4 for every solution to be unbounded.  Now 4 is the average of 2 and 6, the
respective periods of every positive solution of the equations

\[ x_{n}=\frac{1}{x_{n-1}} \ \ \ \mbox{and} \ \ \ x_{n}=\frac{1}{x_{n-3}}. \]

\noindent The right sides of these two equations make up the arguments of (\ref{equation:five}).

\item Finally, for (\ref{equation:six}), the period of either $\{A_{n}\}_{n=0}^{\infty}$ or 
$\{B_{n}\}_{n=0}^{\infty}$ must be a multiple of 5 for every solution to be
unbounded.  Now 5 is the average of 4 and 6, the respective periods of every
positive solution of the equations

\[ x_{n}=\frac{1}{x_{n-2}} \ \ \ \mbox{and} \ \ \ x_{n}=\frac{1}{x_{n-3}}. \]

\noindent The right sides of these two equations make up the arguments of (\ref{equation:six}).

\end{enumerate}

\vspace{.1in}

\noindent Now we note that the number $t+1$ is the average of the respective periods of every positive solution of the equations

\[ x_{n}=\frac{1}{x_{n-1}}, \ \ x_{n}=\frac{1}{x_{n-2}}, \ \ \ldots, \ \ x_{n}=\frac{1}{x_{n-t}}, \]

\noindent where

\[ 
\frac{\displaystyle{2\sum_{\ell =1}^{t} \ \ell}}{t}=\frac{2t(t+1)}{2t}=t+1. \]

\noindent The right sides of these $t$ equations make up the arguments of
(\ref{equation:eight}).



\section{Sufficient Conditions for Boundedness}
\label{section:Sufficient-Conditions-for-Boundedness}

%
%

\noindent In this section, we find sufficient conditions on the periods $p_{i}$
of the sequences $\{A_{n}^{i}\}_{n=0}^{\infty}$ for all $i\in[t]$ such that
every positive solution of (\ref{equation:eight}) is bounded (and persists).
The bulk of the work falls into Lemmas~\ref{lemma1} and \ref{lemma2}.  In the
first lemma, we prove that for any $M\in\ZedP$, we can find sufficiently small
$\epsilon$ so that if some $x_n<\epsilon$, then we know the exact value of
nearly all of the $(t+1)M$ terms preceeding $x_n$.  In the second lemma, we show
that if, in addition, certain of the $t$ sequences $\{A_{n}^{i}\}_{n=0}^{\infty}$
have period relatively prime to $t+1$, then some specific term further along in
the sequence must be at least as large as this $x_n$.  Finally, we show that
every such sequence must be bounded.


\begin{lemma}
\label{lemma1}
Given $M\in\Zed^+$, there exist $r>0$ and $\epsilon>0$ such that if there exists
$N\in\Nats$ with $x_{N}<\epsilon$, then we have
\begin{align}
x_{N-k(t+1)-i}=\frac{A^{t+1-i}_{N-k(t+1)-i-1}}{x_{N-(k+1)(t+1)}}\label{eq:first}\\
x_{N-(k+1)(t+1)}<\epsilon r^{k+1}\label{eq:second}
\end{align}
for all $i\in[t]$ and all $k$ with $0\le k\le M-1$.
\end{lemma}
\begin{proof}
Intuitively, we want to know the exact value of $x_i$ for many of the values of
$i$ in some large range, namely a range of $M(t+1)$ successive values.  Given 
$M$, we find a small value $\epsilon$ such that if $x_N<\epsilon$, then we know
the exact value of each $x_{N-j}$ when $j\in[M(t+1)]$ and
$(t+1)\!\!\not\!|\,j$.  
%
Roughly speaking, we show that if $j\in[M(t+1)]$ and $(t+1)|j$, then
$x_{N-j}$ must be very small; so small, in fact, that for each $i\in[t]$ the
maximum in the definition of $x_{N-j+i}$ is achieved by the term that divides by
$x_{N-j}$.  Throughout the proof, we need a number of inequalities that bound
$\epsilon$ from above.  We do not list these explicitly.  Rather, we
note only that we have a finite number of inequalities that bound $\epsilon$
from above, and the upper bounds are strictly positive.  Thus, we can choose
positive $\epsilon$ that satisfies them all.

By possibly relabeling, we assume $N=0$.  Let $\alpha=\max_{i,j}A^i_j$ and
$\beta=\min_{i,j}A^i_j$ and let $r=\alpha/\beta$.  We primarily want to prove
(\ref{eq:first}), which will satisfy the hypothesis of our next lemma. 
However, for the proof, we need to prove (\ref{eq:second}) as well.  Our proof
is by induction on $k$.  When invoking the induction hypothesis, we will only
assume (\ref{eq:second}).  Thus, our base case
$k=0$ is simply a special case of the general induction step, since by
hypothesis we have $x_0<\epsilon$.

To prove (\ref{eq:first}) it suffices to show that
\begin{align}
x_{-k(t+1)-j} &>
x_{-(k+1)(t+1)}\frac{A^{j-i}_{-k(t+1)-i-1}}{A^{t+1-i}_{-k(t+1)-i-1}}
\label{eq:third}
\end{align}
for all $i\in[t]$ and $i<j<i+(t+1)$, such that $j\ne t+1$.  This inequality
may look daunting, but it is simply saying that when computing $x_{-k(t+1)-i}$,
the term that divides by $x_{-(k+1)(t+1)}$ is larger than the term that divides
by $x_{-k(t+1)-j}$.  We first prove (\ref{eq:second}).  We consider for
$x_{-k(t+1)-1}$ which argument is maximum, and we show this is the argument
that divides by $x_{-(k+1)(t+1)}$.  If instead the maximum argument
divides by $x_{-k(t+1)-j}$, then we get
\begin{align*}
x_{-k(t+1)-1}&=\frac{A^{j-1}_{-k(t+1)-2}}{x_{-k(t+1)-j}}\\
&<
\frac{A^{j-1}_{-k(t+1)-2}}{A^j_{-k(t+1)-1}}\epsilon r^k\\
&<
\frac{A^1_{-k(t+1)-1}}{\epsilon r^k}\\
&<
x_{-k(t+1)-1},
\end{align*}
which is a contradiction. The first inequality 
holds because $x_{-k(t+1)}< \epsilon r^k$ and 
$A^j_{-k(t+1)-1}/x_{-k(t+1)-j}\le x_{-k(t+1)}$, 
the second holds because $\epsilon$ is sufficiently
small, and the third holds because $x_{-k(t+1)}<\epsilon r^k$.
This contradiction implies that
$x_{-k(t+1)-1}=\frac{A^t_{-k(t+1)-2}}{x_{-(k+1)(t+1)}}$.  Rewriting, we have
\begin{align*}
x_{-(k+1)(t+1)}&=\frac{A^t_{-k(t+1)-2}}{x_{k(t+1)-1}}\\
&<
\frac{A^t_{-k(t+1)-2}}{A^1_{-k(t+1)-1}}\epsilon r^k\\
&<
\epsilon r^{k+1}.
\end{align*}
As above, the first inequality holds because $x_{-k(t+1)}<\epsilon r^k$; the
second inequality comes from the definition of $r$.
Thus \eqref{eq:second} holds.

Now we prove \eqref{eq:first}.  To do so, we first prove \eqref{eq:third} for
$i<j<t+1$, and then prove \eqref{eq:third} for $t+1<j<i+(t+1)$.
By hypothesis, $x_{-k(t+1)}<\epsilon r^k$.  For each $j\in[t]$, we get
$\epsilon r^k > x_{-k(t+1)} \ge \frac{A^j_{-k(t+1)-1}}{x_{-k(t+1)-j}}$. 
Cross-multiplying gives
\begin{align*}
x_{-k(t+1)-j}&>\frac{A^j_{-k(t+1)-1}}{\epsilon r^k} \\
&>\epsilon r^{k+1}\frac{A^{j-i}_{-k(t+1)-i-1}}{A^{t+1-i}_{-k(t+1)-i-1}}\\
&>x_{-(k+1)(t+1)}\frac{A^{j-i}_{-k(t+1)-i-1}}{A^{t+1-i}_{-k(t+1)-i-1}},
\end{align*}
where the second inequality holds because $\epsilon$ is sufficiently small and
the third holds by (\ref{eq:second}), which we proved above.  So
we have proved (\ref{eq:third}) for $i<j<t+1$. Now we prove it for
$t+1<j<i+(t+1)$.  The argument is quite similar.

By (\ref{eq:second}) we have $x_{-(k+1)(t+1)}<\epsilon r^{k+1}$.  By
transitivity, for each $j$ with $t+1<j<i+(t+1)$, we get $\epsilon
r^{k+1}>x_{-(k+1)(t+1)}\ge \frac{A^j_{-(k+1)(t+1)-1}}{x_{-(k+1)(t+1)-j}}$. 
Rewriting this, we get
\begin{align*}
x_{-(k+1)(t+1)-j}&>\frac{A^j_{-(k+1)(t+1)-1}}{\epsilon r^{k+1}} \\
&>\epsilon r^{k+1}\frac{A^{j-i}_{-k(t+1)-i-1}}{A^{t+1-i}_{-k(t+1)-i-1}}\\
&>x_{-(k+1)(t+1)}\frac{A^{j-i}_{-k(t+1)-i-1}}{A^{t+1-i}_{-k(t+1)-i-1}}.
\end{align*}
As above, the second inequality holds because $\epsilon$ is sufficiently small
and the third holds by \eqref{eq:second}.  So we have proved
\eqref{eq:third} for all $i<j<i+(t+1)$.  Together with the case above, this
proves \eqref{eq:first}, and thus completes the proof.
\end{proof}

\begin{lemma}
\label{lemma2}
Let $\xn$ be a solution of (\ref{equation:eight}).  Suppose there exists $i$ with
$\gcd(t+1,p_ip_{t+1-i})=1$, and let $P=p_ip_{t+1-i}$.  If there exists $N\ge 0$
such that for all $j\in[t]$ and $k\in[P]$
\begin{align}
x_{N+k(t+1)-j}&=
\max\left\{
\frac{A^1_{N+k(t+1)-j-1}}{x_{N+k(t+1)-j-1}},
\frac{A^2_{N+k(t+1)-j-1}}{x_{N+k(t+1)-j-2}},
\ldots,
\frac{A^t_{N+k(t+1)-j-1}}{x_{N+k(t+1)-j-t}}
\right\} \notag\\
&=\frac{A^{t+1-j}_{N+k(t+1)-j-1}}{x_{N+k(t+1)-(t+1)}}, 
\label{equation:twelve}
\end{align}
then $x_{N+P(t+1)}\ge x_N$.
\end{lemma}

\begin{proof}
First note that for any choice of $P$, we can satisfy the second hypothesis by
(\ref{eq:first}) of Lemma~\ref{lemma1}.  Thus, to apply the present lemma, we
will only need to demonstrate that there
exists $i$ with $\gcd(t+1,p_ip_{t+1-i})=1$.
%
By definition 
\begin{align}
x_{N+k(t+1)}=\max_{1\le j\le t}
\left\{
\frac{A^j_{N+k(t+1)-1}}{x_{N+k(t+1)-j}}
\right\}.
\label{equation:thirteen}
\end{align}
By substituting (\ref{equation:twelve}) into (\ref{equation:thirteen}), we get
\begin{align}
x_{N+k(t+1)}=
x_{N+(k-1)(t+1)}
\max_{1\le j\le t}
\left\{
\frac{A^j_{N+k(t+1)-1}}{A^{t+1-j}_{N+k(t+1)-j-1}}
\right\}.
\label{equation:fourteen}
\end{align}

\noindent
By repeated application of recurrence (\ref{equation:fourteen}) for
all $k\in[P]$, we get \begin{align}
x_{N+P(t+1)}=
x_N\prod_{k=1}^P
\max_{1\le j\le t}
\left\{
\frac{A^j_{N+k(t+1)-1}}{A^{t+1-j}_{N+k(t+1)-j-1}}
\right\}.
\label{equation:fifteen}
\end{align}
Recall that $\gcd(t+1,p_ip_{t+1-i})=1$.  As a result, $t+1$ is an additive
generator of $\Zed/P\Zed$.  Applying this fact to subscripts, we get

\noindent
\begin{align}
  \prod_{k=1}^PA^i_{N+k(t+1)-1}&=\prod_{k=1}^PA^i_k,  \label{equation:sixteen}\\
  \prod_{k=1}^PA^{t+1-i}_{N+k(t+1)-1}&=\prod_{k=1}^PA^{t+1-i}_k.  \label{equation:seventeen}
\end{align}


\noindent
From (\ref{equation:fifteen}), we get that

\begin{align}
x_{N+P(t+1)}\ge
x_N
\prod_{k=1}^P
\frac{A^i_{N+k(t+1)-1}}{A^{t+1-i}_{N+k(t+1)-i-1}}
=x_N\frac{\prod_{k=1}^PA^i_k}{\prod_{k=1}^PA^{t+1-i}_k},
\label{equation:eighteen}
\end{align}
where the inequality follows from the definition of maximum, and the equality
follows from substituting (\ref{equation:sixteen}) and
(\ref{equation:seventeen}).  An analogous argument gives that
$x_{N+P(t+1)}\ge x_N\frac{\prod_{k=1}^PA^{t+1-i}_k}{\prod_{k=1}^PA^i_k}.$
Combining this inequality with (\ref{equation:eighteen}), we get that

\begin{align}
x_{N+P(t+1)}\ge
x_N\max\left\{
\frac{\prod_{k=1}^PA^i_k}{\prod_{k=1}^PA_k^{t+1-i}},
\frac{\prod_{k=1}^PA_k^{t+1-i}}{\prod_{k=1}^PA^i_k}
\right\}
\ge x_N(1) = x_N.
\end{align}
Here the second inequality holds because the arguments to $\max$
are reciprocals of each other (hence one of them is at least 1).
\end{proof}

\begin{lemma}
\label{lemma3}
Let $\xn$ be a solution of (\ref{equation:eight}).  If $\xn$ does not persist,
then there exists a subsequence $\{x_{n_k}\}^{\infty}_{k=1}$ of $\xn$ which
possesses the following properties:
\begin{enumerate}
\item[(i)] $x_{n_{k+1}}<x_{n_k}$ for all $k=1,2,\ldots.$
\item[(ii)] If $n_{k+1}>1+n_k$, then $x_{n_k}\le x_n$ and $x_{n_{k+1}}<x_n$ for each
$k=1,2,\ldots$ and for all $n_k<n<n_{k+1}$.
\item[(iii)] $\lim_{k\to\infty}x_{n_k}=0$.
\end{enumerate}
\end{lemma}
\begin{proof}[Sketch]
This lemma was proved for the case $t=2$ in~\cite{KENTR}; however, that
proof also holds for general $t$.  For completeness, we sketch the proof here.
Since $\xn$ does not persist, it contains a subsequence
$\{x_{n_k}\}_{n=0}^{\infty}$ with $\lim_{k\to\infty}x_{n_k}=0$.  We greedily
take a strictly decreasing subsequence of $\{x_{n_k}\}_{n=0}^{\infty}$; it will
evidently satisfy all three desired properties.
\end{proof}

\begin{theorem}[Bounded Solutions]
Let $\xn$ be a solution of (\ref{equation:eight}).  If there exists $i\in[t]$
with $\gcd(t+1,p_ip_{t+1-i})=1$, then $\{x_n\}_{n=-t}^{\infty}$ is bounded and
persists.
\end{theorem}

\begin{proof} [Sketch]
Assume that the hypothesis holds.  Suppose to the contrary that $\xn$ does not
persist.  By Lemma~\ref{lemma3}, we have a subsequence 
$\{x_{n_k}\}^{\infty}_{k=0}$ for
which the three properties of Lemma~\ref{lemma3} hold.  
Property (iii) states that $\lim_{k\to\infty}x_{n_k}=0$, so
we can apply Lemma~\ref{lemma1}.  The conclusion of Lemma~\ref{lemma1}
satisfies the hypothesis of Lemma~\ref{lemma2}.  Now by Lemma~\ref{lemma2},
there exists $n_k$ such that $x_{n_k}\ge x_{n_k-P(t+1)}$ (where $P=p_ip_{t+1-i}$ as
in Lemma~\ref{lemma2}).
This contradicts the properties of Lemma~\ref{lemma3} as follows.  If
$x_{n_k-P(t+1)}$ is an element of the subsequence
$\{x_{n_k}\}^{\infty}_{k=0}$, then it contradicts Property (i).  Otherwise, 
Property (ii) implies that there exist an integer $s$ such that
$n_s<n_k-P(t+1)$ and $x_{n_s}\le x_{n_k-P(t+1)}$.  But now $x_{n_s}\le
x_{n_k-P(t+1)}\le x_{n_k}$, which again contradicts Property (i).
\end{proof}


\section{Sufficient Conditions for Every Solution to Be Unbounded}
\label{section:Sufficient-Conditions-for-Unboundedness}
In this section, we present the second of our two main results.
We initially show that if the sequences $\An$, $i=1,\ldots,t$, of positive real
numbers, which are not necessarily periodic, satisfy Hypothesis (H), then every
positive solution of (\ref{equation:eight}) is unbounded (and does not
persist).  We then show that Hypothesis (H) is satisfied when $\An$ are
periodic, and certain of them have period a multiple of $t+1$.

\begin{theorem}
\label{unbounded solutions}
If $\An$,
$i=1,2,\ldots,t$, is a set of sequences of positive real numbers
satisfying Hypothesis (H), then every positive solution of
\eqref{equation:eight} 
is unbounded.
\end{theorem}

\begin{proof}
Let $\xn$ be a positive solution of (\ref{equation:eight}) and let
$j\in\{0,1,\ldots,t\}$.
If the constants $S_{A^i}$, $I_{A^i}$ as defined in Hypothesis (H)
are postive for all ${i=1,2,\ldots,t}$, then for all $n\ge 0$, we have the following:

\begin{align*}
x_{(t+1)n+(t+2)+j}&=\max_{1\le i\le
t}\left\{\frac{A^i_{(t+1)n+(t+1)+j}}{x_{(t+1)n+(t+2)-i+j}}\right\} \\
&
=\max_{1\le i\le t}
\left\{\frac{A^i_{(t+1)n+(t+1)+j}}
{
\max_{1\le k\le t}
\left\{\frac{A^k_{(t+1)n+(t+1)-i+j}}{x_{(t+1)n+(t+2)-k-i+j}}
\right\}} 
\right\} \\
&
=\max_{1\le i\le t}
\left\{
{
\min_{1\le k\le t}
\left\{
\frac
{A^i_{(t+1)n+(t+1)+j} x_{(t+1)n+(t+2)-k-i+j}}
{A^k_{(t+1)n+(t+1)-i+j}}
\right\}} 
\right\} \\
&
\le\max_{1\le i\le t}
{
\left\{
\frac
{A^i_{(t+1)n+(t+1)+j} 
}
{A^{t+1-i}_{(t+1)n+(t+1-i)+j}}
\right\}}
x_{(t+1)n+1+j}
\\
&\le
\max_{1\le i\le t}\left\{\frac{S_{A^i}}{I_{A^{t+1-i}}}\right\}x_{(t+1)n+1+j}.
\end{align*}
%
The first inequality comes from letting $k=t+1-i$ in the min, and the second
comes from the definitions of $S_{A^i}$ and $I_{A^{t+1-i}}$.
Now let 
$
\alpha = 
\max_{1\le i\le t}\left\{\frac{S_{A^i}}{I_{A^{t+1-i}}}\right\}.
$
Since the sequences $\An$ satisfy Hypothesis (H), we get $\alpha<1$.  We
have just shown that $x_{(t+1)n+(t+2)+j}\le \alpha x_{(t+1)n+1+j}$ for all
$n\ge 0$, so $x_{(t+1)n+1+j}\le \alpha^n x_{1+j}$.
It follows that $\lim_{n\to \infty}x_{(t+1)n+1+j}=0$.
Furthermore, $\lim_{n\to\infty}x_{(t+1)n+k+j}=\infty$ for all $2\le
k\le t+1$.  Below we give the proof for $k=2$; the other proofs are analogous.
\begin{align*}
x_{(t+1)n+2+j}=&\max_{1\le i\le t+1}\left\{
\frac{A^i_{(t+1)n+1+j}}
{x_{(t+1)n+(2-i)+j}}\right\}\\
\ge &
\frac{A^1_{(t+1)n+1+j}}{x_{(t+1)n+1+j}}\\
\ge & \frac{I_{A^1}}{x_{(t+1)n+1+j}}.
\end{align*}
\aftermath
\end{proof}

\begin{remark}
Observe that the solution in Theorem~\ref{unbounded solutions} is extended periodic.
\end{remark}

\begin{corollary}[Periodic Coefficients]
Let $\An$ be a periodic sequence of positive real numbers with period
$p_i\in\ZedP$ for all $i\in[t]$.
Let $d=t/2$ if $t$ is even and let $d=(t-1)/2$ if $t$ is odd.
If $t$ is odd, then assume also that for $i=(t+1)/2$ we have 
$k_i\in\ZedP$ such that $p_i=(t+1)k_i$ and for all $l_i\in[k_i]$:
\begin{align*} 
A^i_{(t+1)l_i+(t+1)+j}<A^i_{(t+1)l_i+i+j}.
\end{align*}

\noindent
Every positive solution of \eqref{equation:eight}
is unbounded if either of the following holds:
\begin{enumerate}
\item
For each 
$i\in[d]$, 
we have (a) $p_i=(t+1)k_i$, where
$k_i\in\ZedP$, 
and (b) for some $j\in[t+1]$ and for all $l_i\in[k_i]$
we have:
\begin{align*}
A^i_{(t+1)l_i+(t+1)+j}
&<\min\{A^{t+1-i}_1,\ldots,A^{t+1-i}_{p_{t+1-i}}\}\\
&<\max\{A^{t+1-i}_1,\ldots,A^{t+1-i}_{p_{t+1-i}}\}\\
&<A^i_{(t+1)l_i+i+j}.
\end{align*}
\item
For each $i\in[d]$, we have (a) $p_{t+1-i}=(t+1)k_{t+1-i}$, where
$k_{t+1-i}\in\ZedP$, and (b) for some $j\in[t+1]$ and for
all $l_{t+1-i}\in[k_{t+1-i}]$
we have:
\begin{align*}
A^{t+1-i}_{(t+1)l_{t+1-i}+(t+1)+j}
&<\min\{A^i_1,\ldots,A^i_{p_i}\}\\
&<\max\{A^i_1,\ldots,A^i_{p_i}\}\\
&<A^{t+1-i}_{(t+1)l_{t+1-i}+(t+1-i)+j}.
\end{align*}
\end{enumerate}
\end{corollary}
\begin{proof}
$\An$, for all $i\in\{1,2,\ldots,t\}$, satisfies Hypothesis (H).
\end{proof}


\section{Future Goals}
\label{section:Future-Goals}

\noindent We conclude with an open problem and with suggestions for two potential biological applications that could be considered novel for max-type difference equations.
%
%
First, our open problem.

\begin{openproblem}

\noindent Consider the difference equation

\[ x_{n}=\max\left\{\frac{A_{n-1}^{k_{1}}}{x_{n-k_{1}}}, \frac{A_{n-1}^{k_{2}}}{x_{n-k_{2}}}, \ldots, 
\frac{A_{n-1}^{k_{t}}}{x_{n-k_{t}}}\right\}, \ \ n=1, 2, \ldots, \]

\noindent where $t\in\{2, 3, \ldots\}$, $k_{1}, k_{2}, \ldots, k_{t}\in\{1, 2, \ldots\}$,
$\{A_{n}^{k_{i}}\}_{n=0}^{\infty}$ $(i=1, 2, \ldots, t)$ is a periodic sequence of \emph{nonnegative} numbers with period 
$p_{k_{i}}\in\{1, 2, \ldots\}$, and initial conditions are positive.  Find necessary and sufficient conditions on the periods and periodic parameters such that every solution is bounded.  

\end{openproblem}

\vspace{.1in}

\noindent Second, max-type difference equations in essence belong to a larger
group of difference equations, called piecewise-defined difference
equations 
\cite{AGKL, BBCK, F, GL, KOCIC}.  
Nice features of these difference equations which make them
especially suited to serve as models of biological processes and systems
include their ``decision-making" properties with the incorporation of
thresholds and their sometimes eventually periodic or unbounded behavior. 
Max-type difference equations such as (\ref{equation:eight}) have the
additional attractive feature of allowing an arbitrary number of variable
parameters.  
\vspace{.1in}

\noindent Piecewise-defined difference equations have been used as models for
\emph{neural networks}~\cite{C1, C2}, as well as differential
equations with maxima (cf.~\cite{BH}), the counterparts to max-type difference
equations.
%
%
Less frequently, they have been applied to 
\emph{morphogenesis}~\cite{DP, SH}, which investigates 
the origins of growth and shape from the embryo to the full adult.  The study of morphogenesis includes analyzing the occurrence of repetitive patterns of development, for example, zebra stripes.  There is also abnormal morphogenesis, which is seen in the development of cancer in which there is excessive, almost unbounded, growth of tissues.

\vspace{.1in}

\noindent We propose that in the future one might consider max-type equations such as (\ref{equation:eight}) or modifications of (\ref{equation:eight}) as candidates for the modeling of neural networks and/or morphogenesis.





\end{document}